\documentclass[12pt]{amsart}
\usepackage{graphicx}
\usepackage{amsfonts}
\usepackage{epsf}
\usepackage{amssymb}
\usepackage{amsmath}
\usepackage{amscd}
\usepackage{amsthm}
\usepackage{tikz-cd}
\usepackage{pdfpages}
\usepackage{fancyhdr}
\usepackage{setspace}
\usepackage{hyperref}
\usepackage[all]{xy}
\usetikzlibrary{matrix}
\usepackage{verbatim}
\usepackage{enumerate}
\usepackage{mathrsfs}
\usepackage{pinlabel}
\usepackage{lscape}
\usepackage{color}
\usepackage[colorinlistoftodos]{todonotes}
\usepackage{mathtools}

\newtheorem{theorem}{Theorem}
\newtheorem{proposition}[theorem]{Proposition}
\newtheorem{lemma}[theorem]{Lemma}

\newtheorem{corollary}[theorem]{Corollary}

\newtheorem*{question}{Question}

\newcommand{\s}{\mathfrak{s}}
\renewcommand{\t}{\mathfrak{t}}

\newcommand{\Z}{\mathbb{Z}}

\newcommand{\Q}{\mathbb{Q}}

\newcommand{\defeq}{\vcentcolon=}
\def\co{\colon\thinspace}

\makeatletter
\newtheorem*{rep@theorem}{\rep@title}
\newcommand{\newreptheorem}[2]{%
\newenvironment{rep#1}[1]{%
 \def\rep@title{#2 \ref{##1}}%
 \begin{rep@theorem}}%
 {\end{rep@theorem}}}
\makeatother

\newreptheorem{theorem}{Theorem}
\newreptheorem{lemma}{Lemma}
\newreptheorem{question}{Question}
\newreptheorem{corollary}{Corollary}

\begin{document}
\makeatletter
\providecommand\@dotsep{5}
\makeatother
\rhead{\thepage}
\lhead{\author}
\thispagestyle{empty}


\raggedbottom
\pagenumbering{arabic}
\setcounter{section}{0}



\title{Lattices and correction terms}
\author{Kyle Larson}
\address{Alfr\'ed R\'enyi Institute of Mathematics, Budapest, Hungary}
\email{larson@renyi.mta.hu}

\begin{abstract}
Let $L$ be a nonunimodular definite lattice, $L^*$ its dual lattice, and $\lambda$ the discriminant form on $L^*/L$.
Using a theorem of Elkies we show that whether $L$ embeds in the standard definite lattice of the same rank is completely determined by a collection of lattice correction terms, one for each metabolizing subgroup of $(L^*/L, \lambda)$.
As a topological application this gives a rephrasing of the obstruction for a rational homology 3--sphere to bound a rational homology 4--ball coming from Donaldson's theorem on definite intersection forms of 4--manifolds. 
Furthermore, from this perspective it is easy to see that if the obstruction to bounding a rational homology ball coming from Heegaard Floer correction terms vanishes, then (under some mild hypotheses) the obstruction from Donaldson's theorem vanishes too.

\end{abstract}
\maketitle


\begin{section}{Introduction}\label{introduction}
In \cite{Elkies} Elkies showed that every unimodular positive definite lattice $L$ of rank $n$ contains characteristic vectors with square less than or equal to $n$, and if there are no characteristic vectors with square strictly less than $n$ then $L$ is isomorphic to the standard lattice $(\Z^n, I)$.
One can define a \emph{lattice correction term} 
\begin{equation}\label{definition}
d_L = \min \Big\{\frac{\chi^2-n}{4}\Big\},
\end{equation}
where the minimum is over all characteristic vectors $\chi \in L$ (see \cite{Greene2}).
This is well-defined for all positive definite unimodular lattices, and Elkies' result translates to the statement that $d_L \leq 0$ and $d_L = 0$ if and only if $L$ is isomorphic to the standard lattice.
Our first goal is to generalize this to the case where $L$ is not unimodular.
In this setting one can ask whether a definite lattice \emph{embeds} in the standard lattice of the same rank.

Recall that we have a sequence $0 \rightarrow L \rightarrow L^* \xrightarrow{\pi} L^*/L \rightarrow 0$, and $L$ is nonunimodular if and only if the discriminant group $L^*/L$ is non-trivial.
There is a one-to-one correspondence between metabolizers $M < L^*/L$ and unimodular lattices $U$ with $L\subset U \subset L^*$, given by $U \defeq \pi^{-1}(M)$ (Proposition \ref{one to one}).
(Recall that a metabolizer is a subgroup $M$ with $|L^*/L|=|M|^2$ and such that the discriminant form 
$\lambda$ is identically zero on $M$.)
For a metabolizer $M$ we denote the corresponding unimodular lattice $U(M)$.
Then $U(M)$ will necessarily be positive definite of rank $n$, and hence we have a lattice correction term $d_{U(M)}$.
We derive the following as a corollary of Elkies' theorem.

\begin{theorem}\label{lattice theorem}
For $L$ a positive definite lattice of rank $n$, consider the set $D \defeq \{d_{U(M_i)}\}$ of lattice correction terms, where we range over all metabolizers $M_i < L^*/L$. Then $L$ embeds in the standard lattice of rank $n$ if and only if $D$ contains 0.
\end{theorem}

Note that $D$ is a finite set since $L^*/L$ is a finite group, and $D$ is empty (and hence $L$ does not embed in the standard lattice) if there do not exist any metabolizers.
Our main interest in this result is in application to the following question in low-dimensional topology:

\begin{question}[\cite{Kirby}, Problem 4.5]
When does a rational homology 3--sphere bound a rational homology 4--ball?
\end{question}

We are interested in the relationship between two obstructions to a rational homology 3--sphere $Y$ smoothly bounding a rational homology 4--ball.
Suppose $Y$ bounds a smooth positive definite 4--manifold $X$, and we will make the simplifying assumption that $H_1(X)=0$. 
If $Y$ bounds a smooth rational homology ball $W$ as well, we can 
form a smooth, closed, definite 4--manifold $Z = X \cup_{Y} -W$.
By Donaldson's theorem \cite{Donaldson, Donaldson2} on definite intersection forms of smooth, closed 4--manifolds, the lattice $(H_2(Z),Q_Z)$ must be isomorphic to the standard lattice.
It then follows that the lattice $(H_2(X),Q_X)$ must embed in the standard lattice of the same rank.
This is what we call \emph{the obstruction to $Y$ bounding a rational homology ball coming from Donaldson's theorem}.

By Theorem \ref{lattice theorem}, the obstruction coming from Donaldson's theorem is completely determined by a collection of lattice correction terms, one for each metabolizing subgroup of $(H_1(Y),\lambda)$ (in this context $\lambda$ is known as the linking form).
Then work of Ozsv{\'a}th and Szab{\'o} \cite{OSz} shows that the Heegaard Floer corrections terms of $Y$ put bounds on the values of these lattice correction terms (indeed, the definition (\ref{definition}) of lattice correction terms is motivated by properties of Heegaard Floer correction terms, see Section \ref{QHS}).
We use these bounds to show that the vanishing of the Heegaard Floer correction terms $d(Y,\t)$ on a metabolizing subgroup (in fact a slightly weaker condition) implies that $(H_2(X),Q_X)$ embeds in the standard lattice of the same rank. 

\begin{theorem}\label{correction terms theorem}
Suppose $Y$ is a rational homology 3--sphere such that $d(Y,\t) \geq 0$ for all spin$^c$ structures $\t$ with $\textrm{PD}(c_1(\t))$ belonging to some fixed metabolizer $M$ of $H_1(Y)$.
Then if $Y$ bounds a smooth positive definite 4--manifold $X$ with $H_1(X)=0$, the lattice $(H_2(X),Q_X)$ must embed in the standard lattice of the same rank.
\end{theorem}

Furthermore, if the correction terms $d(Y,\t)$ corresponding to a metabolizer $M$ as above are all strictly positive, then $Y$ cannot bound a positive definite 4--manifold $X$ with $H_1(X)=0$ (Proposition \ref{bounds definite}).
Recall that if $Y$ bounds a rational homology ball, then $d(Y,\t) = 0$ for all spin$^c$ structures $\t$ that extend over the rational homology ball.
This can be interpreted in a convenient way if we further assume that $Y$ is a $\Z/2\Z$-homology sphere (so $|H_1(Y)|$ is odd).
In particular, if such a $Y$ bounds a rational homology ball then there exists a metabolizer $M$ such that $d(Y,\t) = 0$ for all spin$^c$ structures $\t$ with $\textrm{PD}(c_1(\t)) \in M$ (see, for example, \cite{HLR}).
This is what we call \emph{the obstruction to a $\Z/2\Z$-homology sphere $Y$ bounding a rational homology ball coming from correction terms},
and hence Theorem \ref{correction terms theorem} shows that (in this context) the obstruction to bounding a rational homology ball coming from Heegaard Floer correction terms is always at least as strong as that coming from Donaldson's theorem.
Note that if $d(Y,\t) = 0$ for all spin$^c$ structures $\t$ with $\textrm{PD}(c_1(\t)) \in M$, then $-Y$ also satisfies the conditions of Theorem \ref{correction terms theorem} (since $d(-Y,\t)=-d(Y,\t)$), and so by reversing orientation we get a statement that also applies to \emph{negative} definite fillings of $Y$.

\begin{corollary}\label{stronger}
If $Y$ is a $\Z/2\Z$-homology sphere on which the correction term obstruction to bounding a rational homology ball vanishes, then for any smooth (positive or negative) definite 4--manifold $X$ with $H_1(X)=0$ and $\partial X = Y$, the lattice $(H_2(X),Q_X)$ must embed in the standard lattice of the same rank.
\end{corollary}

When $Y$ is not a $\Z/2\Z$-homology sphere, the first Chern class mapping is no longer a bijection, and Theorem \ref{correction terms theorem} is less useful. For example, the lens space $L(4,1)$ (which does bound a rational homology ball) does not satisfy the hypotheses of Theorem \ref{correction terms theorem}, since $\textrm{PD}(c_1(\t))$ belongs to the unique metabolizer for \emph{each} spin$^c$ structure on $L(4,1)$, and there exists a spin$^c$ structure whose corresponding correction term is negative.

That there is a close relationship between Donaldson's theorem 
and the correction terms of Heegaard Floer homology was already established in the original paper defining correction terms \cite{OSz}.
Indeed, using the theorem of Elkies mentioned above, Ozsv{\'a}th and Szab{\'o} gave a new proof of Donaldson's theorem using properties of the unique correction term $d(N)$ for an integral homology sphere $N$ (this mirrored another proof in Seiberg-Witten Floer theory \cite{Froyshov}).
Furthermore, they showed that the obstruction to an \emph{integral} homology 3--sphere bounding an \emph{integral} homology 4--ball coming from correction terms is at least as strong as that coming from Donaldson's theorem.
More precisely, suppose $N$ bounds a positive definite 4--manifold $X$.
Then if $d(N)=0$ (which must be the case if $N$ bounds an integral homology ball), $Q_X$ must be isomorphic to the standard form \cite{OSz} (note that Corollary \ref{stronger} is a generalization of this statement).

If we are dealing with a rational rather than integral homology sphere, the two obstructions are slightly more complicated, as we described above. 
The extra complication in the case of Donaldson's theorem is because we have to consider embeddings, rather than isomorphisms, of lattices; in the case of correction terms it is because there is no longer a unique correction term, but rather a collection of correction terms corresponding to the set of spin$^c$ structures on the 3--manifold.

Nonetheless, relations between these two obstructions for rational homology spheres have appeared previously in the literature, usually in more specific contexts.
In \cite{GJ} Greene and Jabuka showed that in the application of these obstructions to showing that certain types of knots (e.g. alternating knots) are not slice, one can view the correction term obstruction as a second-order obstruction after the vanishing of the obstruction coming from Donaldson's theorem (see \cite{GJ} Theorem 3.6 and the preceding exposition).
More recently, Greene \cite{Greene} showed that in certain special cases these two obstructions can be used to achieve the same purpose.
For example, he showed that either obstruction is sufficient to classify which lens spaces $L(p,q)$ with odd $p$ bound rational homology balls (which had been carried out by Lisca \cite{Lisca} using the obstruction from Donaldson's theorem, including those with even $p$).
Indeed, the proof of Theorem \ref{correction terms theorem} is very similar to the ideas presented in \cite[Proposition 2.1]{Greene}. 
In particular, one direction of Greene's argument gives Corollary \ref{stronger} when $H_1(Y)$ is \emph{cyclic}.
Hence the present note can be thought of as a companion to that paper, where here we take a more general and elementary perspective.

\subsection*{Acknowledgements} The author is pleased to thank Marco Golla, Joshua Greene, Brendan Owens, and Sa\v{s}o Strle for helpful conversations and comments. In addition, we thank Marco Golla for sharing Proposition \ref{one to one} with the author, and a referee for helpful suggestions. This research was partially supported by NKFIH Grant K112735.

\end{section}


\begin{section}{lattices}\label{lattices}

First we develop the necessary terminology about lattices (cf. \cite[Section 2]{Greene2}).
In this paper a \emph{lattice} $(L,Q)$ is a finite rank free abelian group $L$ together with a symmetric, bilinear form $Q \co L \times L \rightarrow \Q$.
We will assume that $Q$ is \emph{nondegenerate}, i.e., for every non-zero $x\in L$ there exists some $y \in L$ such that $Q(x,y)\neq 0$.
Usually the form will be understood and we will just say $L$ is a lattice.
If the image of the form lies in $\Z$, then the lattice will be called \emph{integral}.
We will always use $L$ to denote an integral lattice.
An \emph{isomorphism} of lattices is an isomorphism of the free abelian groups that preserves the forms, and an \emph{embedding} of lattices is a monomorphism that preserves the forms.

We say $L$ is \emph{positive definite} if the rank of $Q$ equals its signature, and \emph{negative definite} if the rank of $Q$ equals $-1$ times the signature.
The standard positive definite lattice, or more simply, the \emph{standard lattice} (of rank $n$), is $(\Z^n, I)$.
This means that in a chosen basis the form is represented by the identity matrix.

The form $Q$ extends to a rational valued form on $L \otimes \Q$, and the \emph{dual lattice} $L^*$ is defined as the subset $L^* = \{x \in L \otimes \Q \mid Q(x,y) \in \Z, \forall y\in L\}$.
The quotient $L^*/L$ is called the \emph{discriminant group}, and its order is the  \emph{discriminant} of $L$, denoted disc$(L)$.
If disc$(L) = 1$, then we say $L$ is \emph{unimodular}.
Note that we have a sequence $0 \rightarrow L \rightarrow L^* \xrightarrow{\pi} L^*/L \rightarrow 0$.
We can define a symmetric, bilinear form $\lambda \co  (L^*/L) \times  (L^*/L) \rightarrow \Q/\Z$, called the \emph{discriminant form}, as follows.
For any $x,y \in  L^*/L$, take lifts $\bar{x},\bar{y}\in L^*$ (so $\pi(\bar{x})=x$ and $\pi(\bar{y})=y$), and define $\lambda(x,y) = -Q(\bar{x},\bar{y}) \pmod{1}$.
As mentioned in the introduction, a subgroup $M < L^*/L$ satisfying disc$(L) = |M|^2$ and $\lambda|_{M\times M} \equiv 0$ is called a \emph{metabolizer}.
The following proposition is well-known in various forms (see \cite[Lemma 2.5]{Jabuka}), and is central to our argument.

\begin{proposition}\label{one to one}
There is a one-to-one correspondence between metabolizers of $(L^*/L,\lambda)$ and unimodular integral lattices $U$ with $L \subset U \subset L^*$, given by the assignment $U(M) \defeq \pi^{-1}(M)$, for each metabolizer $M$.
\end{proposition}

\begin{proof}
The map $\pi$ induces a bijection between subgroups of $L^*$ containing $L$ and subgroups of $L^*/L$.
For such a subgroup $U \subset L^*$ containing $L$, 
the rational valued form $Q$ on $L^*$ restricts to an integral form on $U$ if and only if $\lambda$ vanishes on $\pi(U) \times \pi(U)$. 
To see this, recall that for $\bar{x},\bar{y} \in U$, $Q(\bar{x},\bar{y}) \equiv -\lambda(\pi(\bar{x}),\pi(\bar{y})) \pmod{1}$.
Finally, $U$ is unimodular if and only if $U=U^*$, or equivalently, if $[U^*:U]=1$.
Now $$\textrm{disc}(L)=[L^*:L]=[L^*:U^*][U^*:U][U:L],$$ and since $[U:L]=[L^*:U^*]$ by Lemma \ref{indices} below, we have $$\textrm{disc}(L)=[U^*:U]([U:L])^2=[U^*:U]|\pi(U)|^2.$$
Hence $U$ is unimodular if and only if $|\pi(U)|^2 = \textrm{disc}(L)$.

\end{proof}

\begin{lemma}\label{indices}
Let $L'$ be an integral lattice with $L \subset L' \subset L^*$. Then $[L':L]=[L^*:(L')^*]$.
\end{lemma}

\begin{proof}
Let $H = \pi(L')$, so $H \cong L'/L$.
Furthermore let $H^\circ$ denote its \emph{annihilator}, that is, the subgroup of $L^*/L$ consisting of all elements that pair trivially with every element of $H$ under $\lambda$.
Observe that $H^\circ = \pi((L')^*)$. 
We claim that $(L^*/L)/H^\circ \cong H$.
To see this, note that the map $\psi \co L^*/L \rightarrow \textrm{Hom}(L^*/L,\Q/\Z)$ given by $\psi(x) = \lambda(x,\cdot)$ is an isomorphism since $\lambda$ is nondegenerate (see \cite{Wall}, especially Sections 1 and 7).
Indeed, for each $x \in L^*/L$ there exists some $y \in L^*/L$ such that $\lambda(x,y) = 1/n \in \Q/\Z$, where $n$ is the order of $x$.
We get a map $\psi' \co L^*/L \rightarrow \textrm{Hom}(H,\Q/\Z) \cong H$ by restricting the domain of each $\psi(x)$ to $H$.
Then $\psi'$ has image isomorphic to $H$ and kernel $H^\circ$, giving $(L^*/L)/H^\circ \cong H$.
It then follows that
\begin{equation*}
[L':L]=|H|= |(L^*/L)/H^\circ|=[L^*:(L')^*],
\end{equation*}
completing the proof.
\end{proof}

We introduce some additional terminology.
A \emph{characteristic covector} $\chi \in L^*$ is an element such that $Q(\chi, y)\equiv Q(y,y) \pmod{2}$ for all $y \in L$.
Let Char$(L)$ denote the set of characteristic covectors. 
If a characteristic covector $\chi$ actually lies in $L$ (as will always be the case when $L$ is unimodular), we can simply call $\chi$ a characteristic \emph{vector}.
As in the introduction, if $L$ is unimodular and positive definite, we have a well-defined \emph{lattice correction term}
\begin{equation*}
d_L = \min_{\chi \in \textrm{Char}(L)} \Big\{\frac{\chi^2-\textrm{rk}(L)}{4}\Big\}.
\end{equation*}
(\cite{Greene2} contains an extended discussion of this invariant.)
In this language we can state the result of Elkies as follows.

\begin{theorem}[\cite{Elkies}]\label{Elkies}
For $L$ a unimodular positive definite lattice, $d_L \leq 0$ and $d_L = 0$ if and only if $L$ is isomorphic to the standard lattice.
\end{theorem}

Hence the lattice correction term completely determines when a unimodular positive definite lattice is isomorphic to the standard lattice.
We can combine Theorem \ref{Elkies} and Proposition \ref{one to one} to characterize which nonunimodular positive definite lattices embed in the standard lattice.
Let $L$ be such a lattice, and $M < L^*/L$ be a metabolizer.
Since $L$ is positive definite, $U(M)$ is positive definite as well, since $L \subset U(M)$ and both lattices have the same rank.
Hence we can define a set $D \defeq \{d_{U(M_i)}\}$ of lattice correction terms, where we range over all metabolizers $M_i < L^*/L$.
Recall that Theorem \ref{lattice theorem} from the introduction states that $L$ embeds in the standard lattice of the same rank if and only if $D$ contain 0.
We prove this theorem now. 

\begin{proof}[Proof of Theorem \ref{lattice theorem}]
If $D$ contains 0, then some $U(M)$ satisfies $d_{U(M)}=0$.
By Theorem \ref{Elkies}, $U(M)$ is isomorphic to the standard lattice.
Since $L \subset U(M)$, one direction of the proof is finished.

In the other direction, suppose $L$ embeds in the standard lattice $(\Z^n,I)$ of the same rank.
Hence we can suppose $L \subset \Z^n$, and tensoring with $\Q$ shows that $(\Z^n,I) \subset L^* \subset L \otimes \Q$.
By Proposition \ref{one to one}, $(\Z^n,I) = U(M)$ for some metabolizer $M$, and Theorem \ref{Elkies} implies that $d_{U(M)} = 0$. This completes the other direction of the proof.
\end{proof}

Let $n$ denote the rank of $L$ (and $U(M)$).
Recall that $d_{U(M)}$ is defined as
\begin{equation}\label{lattice correction term 2}
d_{U(M)} = \min_{\chi \in \textrm{Char}(U(M))} \Big\{\frac{\chi^2-n}{4}\Big\}.
\end{equation}

Since $L \subset U(M) \subset L^*$, $\textrm{Char}(U(M)) \subset \textrm{Char}(L)$. 
Indeed, $\textrm{Char}(U(M))$ is a subset of those characteristic covectors of $L^*$ that map to elements of $M$ under the projection $\pi$.
Hence from (\ref{lattice correction term 2}) we obtain
\begin{equation}\label{lattice correction term 3}
d_{U(M)} \geq \min_{\substack{\chi \in \textrm{Char}(L)\\ \pi(\chi) \in M}} \Big\{\frac{\chi^2-n}{4}\Big\}.
\end{equation}
This will be useful in the next section.
Note that it is possible to show we have equality in (\ref{lattice correction term 3}) if disc$(L)$ is odd.
\end{section}

\begin{section}{Rational homology spheres and correction terms}\label{QHS}

We now turn to the topological application discussed in the introduction.
Let $Y$ be a rational homology 3--sphere that bounds a smooth positive definite 4--manifold $X$ with $H_1(X)=0$. 
By the long exact sequence of the pair $(X,Y)$ we get the presentation
\begin{equation}\label{ses}
0 \rightarrow H_2(X) \rightarrow H_2(X,Y) \rightarrow H_1(Y) \rightarrow 0.
\end{equation}
Under suitable choices of bases, the map $H_2(X) \rightarrow H_2(X,Y)$ is given by the matrix representing the intersection form $Q_X$ (see, for example, \cite[Exercise 5.3.13 (f)]{GS}).
Furthermore, if we let $L$ denote the lattice $(H_2(X),Q_X)$, the dual lattice $L^*$ is identified with $(H_2(X,Y),Q_X^{-1})$, and (\ref{ses}) becomes
\begin{equation}
0 \rightarrow L \xrightarrow{Q_X} L^* \xrightarrow{\pi} L^*/L \rightarrow 0.
\end{equation}
In this context the discriminant form is called the \emph{linking pairing} $\lambda$ on $H_1(Y) \cong L^*/L$, and is defined by $\lambda(x,y) = -(Q_X)^{-1}(\pi^{-1}(x),\pi^{-1}(y)) \pmod{1}$.
Note that this is independent of the choice of the 4-manifold $X$.

As explained in the introduction, a consequence of Donaldson's theorem is that if $Y$ smoothly bounds a rational homology ball, then the lattice $L = (H_2(X),Q_X)$ embeds in the standard lattice of the same rank.
By Theorem \ref{lattice theorem}, this condition is completely determined by the collection of lattice correction terms $\{d_{U(M_i)}\}$, where we range over metabolizers of $(H_1(Y),\lambda)$.
These lattice correction terms are in turn bounded by the Heegaard Floer correction terms of $Y$, as we now describe.
Recall that in Ozsv{\'a}th and Szab{\'o}'s Heegaard Floer homology, correction terms are rational valued invariants of spin$^c$ rational homology spheres that are preserved under spin$^c$ rational homology cobordism.
For $Y$ with spin$^c$ structure $\t$, the corresponding correction term is denoted $d(Y,\t)$.
We have the following important inequality.

\begin{theorem}[\cite{OSz}]
Let $Y$ be a rational homology sphere that bounds a positive definite 4--manifold $X$. 
If $\s$ is a spin$^c$ structure on $X$ with $\s|_Y = \t$, then
\begin{equation}\label{bound}
\frac{1}{4}(c_1(\s)^2- \textup{rk}(H_2(X))) \geq d(Y,\t).
\end{equation}
\end{theorem}

\begin{figure}
\begin{tikzcd}[row sep=scriptsize, column sep=scriptsize]
  & 0 & 0\\
  \textrm{Spin}^c(Y) 
   \arrow{r}{c_1}  
  & H^2(Y)
 	\arrow{u}
     	\arrow{r}{\textrm{PD}}  
  & H_1(Y) 
  	 \arrow[swap]{u}\\
  \textrm{Spin}^c(X)
          \arrow{u}
          \arrow{r}{c_1}
& H^2(X)
	\arrow{u}
	\arrow{r}{\textrm{PD}} 
& H_2(X,Y)
	\arrow{u}\\
& H^2(X,Y)
	\arrow{r}{\textrm{PD}} 
	\arrow{u}
& H_2(X) 
	\arrow{u}
 \end{tikzcd}
 \caption{A commutative diagram.}
 \label{diagram1}
\end{figure}

Now we relate this to lattices.
The first Chern class mapping and Poincar{\'e} duality provide a bijection between spin$^c$ structures on $X$ and characteristic covectors in $H_2(X,Y)$ (\cite[Proposition 2.4.16]{GS}).
Under this bijection, a spin$^c$ structure $\s$ on $X$ that extends a spin$^c$ structure $\t$ on $Y$ corresponds to a characteristic covector $\chi$ in $H_2(X,Y) = L^*$, such that $\pi(\chi) = \textrm{PD}(c_1(\t))$. 
(See Figure \ref{diagram1}.)
Then (\ref{bound}) implies that $\frac{1}{4}(\chi^2 -\textrm{rk}(H_2(X))) \geq d(Y,\t)$ for each such $\chi$. (Other applications of these bounds can be found in \cite{OwensStrle-bounds} and \cite{GW}.)
For a metabolizer $M$, we can combine this with the inequality (\ref{lattice correction term 3}) and Theorem \ref{Elkies} 
to obtain
\begin{equation}\label{big guy}
0 \geq d_{U(M)} \geq \min_{\substack{\chi \in \textrm{Char}(L) \\ \pi(\chi) \in M}} \Big\{\frac{\chi^2-\textrm{rk}(H_2(X))}{4}\Big\} \geq \min_{\substack{\t  \in \textrm{Spin}^c(Y)\\ \textrm{PD}(c_1(\t))\in M}} \Big\{d(Y,\t)\Big\}.
\end{equation}
Note that we would have a contradiction if there exists a metabolizer $M$ for $H_1(Y)$ with
\begin{equation*}
\min_{\substack{\t  \in \textrm{Spin}^c(Y)\\ \textrm{PD}(c_1(\t))\in M}} \Big\{d(Y,\t)\Big\} > 0,
\end{equation*}
and so such a $Y$ cannot bound a smooth positive definite 4--manifold $X$ with $H_1(X)=0$.
We record this here as a proposition.
Note that this generalizes a theorem for integral homology spheres \cite[Corollary 9.8]{OSz}, and for rational homology spheres there are similar results by Owens and Strle \cite{OwensStrle-char} (see Theorem 2 and Proposition 5.2).

\begin{proposition}\label{bounds definite}
Suppose a rational homology sphere $Y$ has a metabolizer $M$ for  $(H_1(Y),\lambda)$ for which $d(Y,\t)>0$ for each spin$^c$ structure $\t$ with $\textrm{PD}(c_1(\t))\in M$.
Then $Y$ cannot bound a smooth positive definite 4--manifold $X$ with $H_1(X)=0$.
\end{proposition}

We can now prove the second theorem from the introduction.

\begin{proof}[Proof of Theorem \ref{correction terms theorem}]
Recall we are assuming that $Y$ is a rational homology 3--sphere that bounds a positive definite 4--manifold $X$ with $H_1(X)=0$, and that there exists a metabolizer $M$ of $H_1(Y)$ such that $d(Y,\t) \geq 0$ for all spin$^c$ structures $\t$ with $\textrm{PD}(c_1(\t)) \in M$.
By Proposition \ref{bounds definite}, there must be at least one such $\t$ such that $d(Y,\t) = 0$, and hence
\begin{equation*}
\min_{\substack{\t  \in \textrm{Spin}^c(Y)\\ \textrm{PD}(c_1(\t))\in M}} \Big\{d(Y,\t)\Big\} = 0.
\end{equation*}
Then equation (\ref{big guy}) implies that $d_{U(M)} = 0$, and by Theorem \ref{lattice theorem} the lattice $(H_2(X),Q_X)$ must embed in the standard lattice of the same rank.
\end{proof}
\end{section}

\begin{section}{Examples}

Finally we give a couple of examples to illustrate these ideas.
First we consider $(+9)$-surgery on the left-handed trefoil, $S^3_9(T_{-2,3})$.
Let $Y$ denote this 3--manifold.
The correction terms of surgeries on torus knots can be computed readily by combining work of \cite{NW} and \cite{BL} (see, for example, \cite{AG}).
Then for the unique metabolizer of $H_1(Y)$, one can check that the corresponding correction terms are $\{2,0,0\}$.
Since one of these is nonzero, the correction term obstruction shows that $S^3_9(T_{-2,3})$ does not bound a rational homology ball.
On the other hand, Theorem \ref{correction terms theorem} states that for every positive definite 4--manifold $X$ with $H_1(X)=0$, the lattice $(H_2(X),Q_X)$ must embed in the standard lattice of the same rank.
Indeed it easy to check this condition for the two obvious positive definite 4--manifolds bounded by $Y$: the 2-handlebody given by the trace of the surgery, and the canonical definite plumbing associated to $S^3_9(T_{-2,3})$ as a Seifert fibered space.
However, $Y$ also bounds a \emph{negative} definite 4--manifold with trivial first homology (see \cite{OwensStrle-torus}), with intersection form 
\[
\begin{bsmallmatrix}
 -1 & 0 & 0 & 0 & 0 & 0 & 0 & 0 & 0  \\
 0 & -2 & 1 & 0 & 0 & 0 & 0 & 0 & 0 \\
 0 & 1 & -2 & 1 & 0 & 0 & 0 & 0 & 0  \\
 0 & 0 & 1 & -2 & 1 & 0 & 0 & 0 & 0  \\
 0 & 0 & 0 & 1 & -2 & 1 & 0 & 0 & 0  \\
 0 & 0 & 0 & 0 & 1 & -2 & 1 & 0 & 0  \\
 0 & 0 & 0 & 0 & 0 & 1 & -2 & 1 & 0  \\
 0 & 0 & 0 & 0 & 0 & 0 & 1 & -2 & 1  \\
 0 & 0 & 0 & 0 & 0 & 0 & 0 & 1 & -2  \\
\end{bsmallmatrix}.
\]
Since the corresponding lattice does not embed in the standard negative definite lattice of rank 9, we see that the obstruction coming from Donaldson's theorem can also be used to show that $Y$ does not bound a rational homology ball.
This suggests the following question.
\begin{question}
Does there exist a rational homology sphere $Y$ for which the correction term obstruction does not vanish, but for any positive \emph{or negative} definite 4--manifold bounded by $Y$, the associated lattice must embed in the standard lattice of the same rank?
\end{question}

Next we use Proposition \ref{bounds definite} to show that the connected sum of the Poincar\'e homology sphere $\Sigma(2,3,5)$ (oriented as the boundary of the \emph{negative} $E_8$ plumbing) and the Seifert fibered space $Y(1;\frac{3}{2},\frac{21}{4},\frac{50}{7})$ does not bound a positive definite 4-manifold with trivial first homology. We label this manifold $Z \defeq \Sigma(2,3,5) \# Y(1;\frac{3}{2},\frac{21}{4},\frac{50}{7})$.
Note that we chose this example because the similar obstructions of Owens and Strle mentioned above (\cite{OwensStrle-char} Theorem 2 and Proposition 5.2) do not apply to $Z$.
The correction terms for each of these manifolds can be computed algorithmically using results of \cite{OSz-plumbed}. 
Using the fact that correction terms add over connected sums, we compute that the correction terms for $Z$ are $\{2,-\frac{4}{9},\frac{2}{9},2,\frac{8}{9},\frac{8}{9},2,\frac{2}{9},-\frac{4}{9}\}$, where this set is identified with $H_1(Z) \cong \Z/9\Z$ in the obvious way.
Hence the correction terms corresponding to the metabolizer are $\{2,2,2\}$, and so Proposition \ref{bounds definite} implies that $Z$ cannot bound a positive definite 4-manifold with trivial first homology.
We do not know if $Z$ bounds a negative definite 4-manifold.

\end{section}


\bibliographystyle{amsalpha}
\bibliography{lattices.bib}

\end{document}